\documentclass[a4paper,11pt,reqno,twoside]{amsart}

\usepackage{amsmath}
\usepackage{amsfonts}
\usepackage{amssymb}
\usepackage{amsthm}
\usepackage{color}
\usepackage{ifpdf}
\usepackage{array}
\usepackage{mathtools}
\usepackage{url}
\usepackage{multirow,bigdelim}

\numberwithin{equation}{section}
 
\newtheorem{theorem}{Theorem}[section] 
 
\newtheorem{prop}{Proposition}[section]

\newcommand*{\R}{\mathbb{R}}
\newcommand*{\Q}{\mathbb{Q}}
\newcommand*{\Z}{\mathbb{Z}}
\newcommand*{\N}{\mathbb{N}}

\newcommand{\comment}[1]{}
\title[Integral operators arising from the Riemann zeta function]%
      {Integral operators arising from the Riemann zeta function} 
\author[M. Suzuki]{Masatoshi Suzuki}
\subjclass[2010]{
11M06, 
45A05, 
33B15. 
}
\keywords{}
\AtBeginDocument{%
\begin{abstract}
In this paper we have two issues coming from the same background.  
The first one is to describe a certain ratio of Fredholm determinants of integral operators 
arising from the Riemann zeta function by using the solution of a single integral equation. 
The second one is to introduce a new integral operator arising from the Riemann zeta function 
and to study its basic analytic properties.
\end{abstract}
\maketitle
}

\begin{document}

\begin{center}
{\it Dedicated to Professor Kohji Matsumoto \\ at the occasion of his 60th Birthday}
\end{center}

\section{Introduction}

The Riemann xi-function $\xi(s)=2^{-1}s(s-1)\pi^{-s/2}\Gamma(s/2)\zeta(s)$ 
is an entire function 
satisfying the functional equation $\xi(s)=\xi(1-s)$, 
where $\zeta(s)$ is the Riemann zeta-function. 
In terms of $\xi(s)$, 
the Riemann hypothesis (RH, for short) asserts that all zeros of $\xi(1/2-iz)$ are real. 
In \cite{La06}, Lagarias pointed out a relationship between the RH and the theory of de Branges spaces 
which are reproducing kernel Hilbert spaces consisting of entire functions. 
The structure of subspaces of a given de Branges space 
is determined by a ${\rm Sym}_2(\R)$-valued function $H(t)$ on an interval 
which is called a Hamiltonian. 
Put $E_\zeta(z)=\xi(1/2-iz)+\xi'(1/2-iz)$,   
and suppose the RH and the simplicity of all zeros. 
Then $E_\zeta$ generates the de Branges space $\mathcal{B}(E_\zeta)$.    
As suggested in \cite{La06},  
the problem 
an explicit form of the Hamiltonian of $\mathcal{B}(E_\zeta)$ 
is interesting and important for the study of the RH. 
However, it is as difficult as the case of general de Branges spaces.  
Therefore, we studied the family
\[
E_\zeta^{\omega,\nu}(z) = \xi(\tfrac{1}{2}+\omega-iz)^\nu, \quad 
(\omega,\nu) \in \R_{>0} \times \Z_{>0} 
\] 
instead of $E_\zeta$ in \cite{Su}. 
Under the RH, each $E_\zeta^{\omega,\nu}$ generates the de Branges space $\mathcal{B}(E_\zeta^{\omega,\nu})$, 
where the simplicity of zeros is unnecessary. 
An advantage of $E_\zeta^{\omega,\nu}$ 
is that the Hamiltonian $H_\zeta^{\omega,\nu}(t)$ of $\mathcal{B}(E_\zeta^{\omega,\nu})$ 
can be constructed explicitly as follows 
for each $(\omega,\nu) \in \R_{>0} \times \Z_{>0}$ satisfying $\omega \nu>1$.   

For $t \geq 0$, we define the operator $\mathsf{K}[t](=\mathsf{K}_\zeta^{\omega,\nu}[t])$ on $L^2(-\infty,t)$ 
by
\begin{equation} \label{eq_180328_01}
(\mathsf{K}[t]f)(x) = \mathbf{1}_{(-\infty,t]}(x) \int_{-\infty}^{t} K(x+y)f(y) \, dy, 
\end{equation}
where $\mathbf{1}_A(x)$ is the characteristic function of a subset $A \subset \R$, 
and the integral kernel $K(x):=K_\zeta^{\omega,\nu}(x)$ is defined by its Fourier integral as follows: 
\begin{equation} \label{eq_180321_02}
\left( \frac{\xi(s-\omega)}{\xi(s+\omega)} \right)^\nu 
= \int_{-\infty}^{\infty} K_\zeta^{\omega,\nu}(x) e^{izx} \, dx,  
\quad s=\frac{1}{2}-iz 
\end{equation}
for $x \geq 0$ and $K(x)=0$ for $x<0$. 
The RH implies that $\det(1 \pm \mathsf{K}[t]) \not=0$ for every $t \geq 0$ (\cite[Prop. 4.4]{Su}). 
Thus,   
\[
m(t) := \frac{\det(1 + \mathsf{K}[t])}{\det(1 - \mathsf{K}[t])}
\]
defines a function of $t \in [0,\infty)$, 
where ``$\det$'' stands for the Fredholm determinant. 
Then we proved that the diagonal matrix-valued function 
$H(t):={\rm diag}(m(t)^{-2}, m(t)^2)$ 
is the Hamiltonian of the de Branges space $\mathcal{B}(E_\zeta^{\omega,\nu})$ (\cite[Thms. 2.2 and 2.3]{Su}). 
The function $m(t)$ satisfies the formula  
\begin{equation} \label{eq_180403_01}
m(t) = \exp\left( \int_{0}^{t}(\phi^+(\tau,\tau)+\phi^{-}(\tau,\tau)) \, d\tau \right), 
\end{equation}
where $\phi^\pm(t,x)$ consists of the unique solutions of the integral equations 
\begin{equation} \label{IE_00}
\phi^\pm(t,x) \pm \int_{-\infty}^{t}K(x+y)\phi^\pm(t,y) \, dy = K(x+t)  
\end{equation}
on $L^2(-\infty,t)$ (\cite[(3.32), (3.35)]{Su}). 
The solutions $\phi^\pm(t,x)$ are also important 
because their extensions to functions of $x$ on $\R$ describe the solution of the canonical system attached to $H(t)$. 

In this paper, we consider the integral equations
\begin{equation} \label{IE_01}
\Phi(t,x) + \int_{-\infty}^{t}K(x+y) \Phi(t,y) \, dy = \mathbf{1}_{[-t,t]}(x)
\end{equation}
and 
\begin{equation} \label{IE_01_b}
\Psi(t,x) - \int_{-\infty}^{t}K(x+y) \Psi(t,y) \, dy = \mathbf{1}_{[-t,t]}(x)
\end{equation}
instead of the pair of integral equations \eqref{IE_00}.  
Then, $m(t)$ and $\phi^\pm(t,x)$ are obtained from the solution of one of these single equation as follows.    

\begin{theorem} \label{thm_1}  
Let $\tau>0$ be a real number such that $\det(1 \pm \mathsf{K}[t]) \not=0$ for every $0 \leq t <\tau$ for the operator \eqref{eq_180328_01}. 
Then the unique solution $\Phi(t,x)$ (resp. $\Psi(t,x)$) of the integral equation \eqref{IE_01} (resp. \eqref{IE_01_b}) is a continuous function of $x$, 
$\Phi(t,t)\,(:=\lim_{x \to t^{-}}\Phi(t,x)) \not=0$ 
(resp. $\Psi(t,t)\,(:=\lim_{x \to t^{-}}\Psi(t,x)) \not=0$) for every $0 <t <\tau$, 
and the following formulas hold for $0 < t <\tau$: 
\begin{equation} \label{formula_02}
\phi^+(t,x) = - \frac{1}{\Phi(t,t)} \frac{\partial}{\partial t}\Phi(t,x) = \frac{1}{\Psi(t,t)} \frac{\partial}{\partial x}\Psi(t,x), 
\end{equation}
\begin{equation} \label{formula_02_b}
\phi^-(t,x) = - \frac{1}{\Phi(t,t)} \frac{\partial}{\partial x}\Phi(t,x) = \frac{1}{\Psi(t,t)} \frac{\partial}{\partial t}\Psi(t,x), 
\end{equation}
\vspace{-14pt}
\begin{equation} \label{formula_01}
m(t) = \frac{1}{\Phi(t,t)} = \Psi(t,t).  
\end{equation}
\end{theorem}

As for assumptions of Theorem \ref{thm_1}, 
it is proved unconditionally that there exists $\tau>0$ such that 
$\det(1 \pm \mathsf{K}[t])\not=0$ for every $0 \leq t < \tau$, 
and $\tau=\infty$ under the RH (\cite[Props. 4.2 and 4.4]{Su}). 
\comment{
On the other hand, if $K(x)=0$ for $x<2t_0$ for some $t_0>0$, 
$\mathsf{K}[t_0]$ is the zero operator and $\phi^\pm(t,x)=0$. 
So the assumption for $\tau_1$ is reasonable. 
It is expected that we can take $\tau_1>0$ arbitrary small, 
but it is not yet proved rigorously.      
}

Now, we change the issue to the second one. 
As proved in \cite[Thm. 2,4]{Su}, 
an equivalent condition for the RH is stated by using the family of operators $\mathsf{K}_\zeta^{\omega,\nu}[t]$. 
However, it would better, if it could be improved finding an equivalent condition 
using only one single operator avoiding parameters $\omega$ and $\nu$. 
As an attempt to achieve such a result, we consider the following matters. 
If $\nu=\theta/\omega$, 
the left-hand side of \eqref{eq_180321_02} has the limit
\[
\lim_{\omega \to 0} \left(\frac{\xi(s-\omega)}{\xi(s+\omega)}\right)^{\theta/\omega}
=\exp\left(-2\theta \frac{\xi'}{\xi}(s) \right).
\]
Then we expect that the kernel $K_\theta(x)$ of the Fourier integral formula  
\begin{equation} \label{eq_180323_01}
\exp\left(-2\theta \frac{\xi'}{\xi}(s) \right) = \int_{-\infty}^{\infty} K_\theta(x) e^{izx} dx, 
\quad s = \frac{1}{2} -iz 
\end{equation}
plays a role similar to $K_\zeta^{\omega,\nu}(x)$ of \eqref{eq_180321_02}. 
In fact, $K_\theta(x)$ satisfies the following properties  
corresponding to the latter four of five properties (K1)$\sim$(K5) required for $K_\zeta^{\omega,\nu}(x)$ in \cite{Su}.

\begin{theorem} \label{thm_2}
If $\theta>1$, $K_\theta(x)$ of \eqref{eq_180323_01} has the following properties: 
\begin{enumerate}

\item[(K-ii)] $K_\theta(x)$ is a real-valued continuous function on $\R$ such that 
$K_\theta(x) \ll \exp(x/2)$ as $x \to +\infty$ and \eqref{eq_180323_01} holds for $\Im(z) > 1/2$, 
\item[(K-iii)] $K_\theta(x)=0$ for $x<0$, 
\item[(K-iv)] $K_\theta(x)$ is continuously differentiable 
on $\R \setminus \{\log n\,|\, n \in \N\}$ 
and $|K_\theta^\prime(x)|$ is locally integrable on $\R$,
\item[(K-v)] there exists $0<\tau \leq \infty$ such that $\det(1\pm \mathsf{K}_\theta[t])\not=0$ for $0 \leq t < \tau$, 
where $\mathsf{K}_\theta[t]$ is the operator defined by 
\eqref{eq_180328_01} for $K_\theta(x)$.  
\end{enumerate}
\end{theorem}

We make some more comments on the function $K_\theta(x)$. 
We have 
\[
\exp\left(-2\theta \frac{\xi'(s)}{\xi(s)}\right)
= \exp\left(-2\theta \frac{\gamma'(s)}{\gamma(s)}\right)\exp\left(-2\theta \frac{\zeta'(s)}{\zeta(s)}\right)
\]
if we put $\gamma(s)=2^{-1}s(s-1)\pi^{-s/2}\Gamma(s/2)$.  
The right factor on the right-hand side has the Dirichlet series expansion 
\begin{equation} \label{eq_180323_02}
\exp\left( -2\theta\frac{\zeta'}{\zeta}(s) \right) = \sum_{n=1}^{\infty} \frac{\lambda_{\theta}(n)}{n^s} 
\end{equation}
endowed with multiplicative coefficients $\lambda_\theta(n)$ 
for $\Re(s)>1$. Thus, the Fourier integral formula 
\begin{equation} \label{eq_180323_03}
\exp\left( -2\theta  \frac{\gamma'}{\gamma}(s) \right)
= \int_{-\infty}^{\infty} g_\theta(x) \,  e^{izx} dx,  \quad s = \frac{1}{2}-iz, 
\end{equation}
shown later, implies the series representation  
\begin{equation} \label{eq_180329_05}
K_\theta(x) = \sum_{n=1}^{\infty} \frac{\lambda_\theta(n)}{\sqrt{n}} g_\theta(x-\log n)  
\end{equation}
consisting of the ``non-archimedean'' or ``arithmetic'' part $\lambda_\theta$ 
and the ``archimedean'' part $g_\theta$.
The Dirichlet series \eqref{eq_180323_02} and its coefficients $\lambda_\theta(n)$ 
are studied in detail by Ihara~\cite{Ih} and Ihara--Matsumoto~\cite{IM01, IM02} 
to investigate the value distribution of $\zeta'/\zeta$. 
Therefore, to prove Theorem \ref{thm_2}, we mainly study $g_\theta(x)$ in \S3.  
\smallskip

As shown in \S3, $g_\theta(x)=0$ for $x<0$, 
thus the sum on the right-hand side of \eqref{eq_180329_05} is finite 
for a bounded range $0 \leq x \leq x_0$. 
Further, we show in \S4 that $g_\theta(x)$ is well approximated 
by a series consisting of Bessel functions in such a range. 
Combining such approximations of $g_\theta(x)$ with \eqref{eq_180329_05}, 
we can calculate $m(t)$, $\Phi(t,x)$ and $\phi^\pm(t,x)$ easily at least in a computational sense 
(cf.~\cite{Bo}). 
\medskip

\noindent
{\bf Acknowledgments}~
This work was supported by KAKENHI (Grant-in-Aid for Scientific Research (C)) No. 17K05163. 

%
%

\section{Proof of Theorem \ref{thm_1}} 

%
%

{\bf Proofs of \eqref{formula_02} and \eqref{formula_02_b}} \quad  
Formulas in \eqref{formula_02} and \eqref{formula_02_b} are proved by a way similar to 
Krein \cite{Kr}. 
First we prove $\Phi(t,t)\not=0$ for $0 < t <\tau$. 
The function $\Phi(t,x)$ is a continuous function of $x$ on $[-t,t]$ by the continuity of $K(x)$ and $\mathbf{1}_{[-t,t]}(x)$. 
Differentiating \eqref{IE_01} with respect to $x$, we have 
\begin{equation} \label{eq_180330_04}
\frac{\partial}{\partial x}\Phi(t,x) + \int_{-\infty}^{t} K'(x+y) \Phi(t,y) \, dy =0
\end{equation}
for $-t<x<t$. This shows that $(\partial/\partial x)\Phi(t,x)$ is a continuous function of $x$ on $[-t,t]$, 
since $|K'(x)|$ is locally integrable (\cite[Prop. 4.1]{Su}). Thus, $(\partial/\partial x)\Phi(t,x) \in L^2(-t,t)$. 
Applying integration by parts to \eqref{eq_180330_04}, 
\begin{equation} \label{eq_180321_01}
\frac{\partial}{\partial x}\Phi(t,x) + K(x+t)\Phi(t,t) - \int_{-\infty}^{t}K(x+y) \frac{\partial}{\partial y}\Phi(t,y) \, dy =0.  
\end{equation}
Therefore, if we suppose that $\Phi(t,t)=0$, 
\begin{equation*} 
\frac{\partial}{\partial x}\Phi(t,x) - \int_{-\infty}^{t}K(x+y) \frac{\partial}{\partial y}\Phi(t,y) \, dy =0.  
\end{equation*}
This asserts that $(\partial/\partial x)\Phi(t,x)$ is a solution of the homogeneous equation $(1-\mathsf{K}[t])f=0$ on $L^2(-t,t)$, 
and thus $(\partial/\partial x)\Phi(t,x)=0$. If $\Phi(t,x)=c$ for $-t < x < t$,  
$
c \left(1 + \int_{0}^{x+t}K(y)\,dy \right) = 1. 
$
Hence,  $K(x)=0$ on $[0,2t)$ which implies that 
$\Phi(t,x) = \mathbf{1}_{[-t,t]}(x)$ for $-t<x<t$ by \eqref{IE_01}. 
Therefore, $\Phi(t,t)=1$ by the continuity of $\Phi(t,x)$ for $x$. 
This is a contradiction. 
Similar arguments also prove $\Psi(t,t)\not=0$ for $0< t <\tau$.
\medskip

Equation \eqref{eq_180321_01} implies that 
$-((\partial/\partial x)\Phi(t,x))/\Phi(t,t)$ solves 
\eqref{IE_00} for the minus sign. 
Hence the uniqueness of the solution concludes the first equality of \eqref{formula_02_b}. 
The second equality of \eqref{formula_02} is also proved in the same way. 
On the other hand, by differentiating \eqref{IE_01} with respect to $t$, 
\[
\frac{\partial}{\partial t}\Phi(t,x) +K(x+t)\Phi(t,t) + \int_{-\infty}^{t}K(x+y)\frac{\partial}{\partial t}\Phi(t,y)\, dy = 0. 
\] 
This shows that $-(
(\partial/\partial t)\Phi(t,x))/\Phi(t,t)$ solves 
\eqref{IE_00} for the plus sign. 
Hence the uniqueness of the solution concludes the first equality of \eqref{formula_02}. 
The second equality of \eqref{formula_02_b} is also proved in the same way. 
\hfill $\qed$
\medskip

%
%

{\bf Proof of \eqref{formula_01}} \quad 
Taking $x = t$ in equation \eqref{IE_01} and then differentiating it with respect to $t$, 
\[
\aligned 
0 
= & \, \frac{d}{dt}(\Phi(t,t)) + 2 K(2t)\Phi(t,t) \\
\quad & - \int_{-\infty}^{t}K(t+y)\frac{\partial}{\partial y}\Phi(t,y) \, dy 
+ \int_{-\infty}^{t}K(t+y)\frac{\partial}{\partial t}\Phi(t,y) \, dy. 
\endaligned 
\]
Using \eqref{formula_02} on the right-hand side, 
\begin{equation} \label{eq_180321_03}
\aligned 
\frac{d}{dt}(\Phi(t,t)) + & \,2 K(2t)\Phi(t,t) \\ 
& - \Phi(t,t) \int_{-\infty}^{t}K(t+x)(\phi^+(t,x) - \phi^-(t,x)) \, dx =0.
\endaligned 
\end{equation}
On the other hand, by the proof of \cite[Thm. 6.1]{Su}, we have 
\[
\aligned 
\frac{1}{2}&(\phi^+(t,x) + \phi^-(t,x)) \\ &  = K(x+t) - \int_{-\infty}^{t}K(x+y)\frac{1}{2}(\phi^+(t,y) - \phi^-(t,y))\,dy.  
\endaligned 
\]
Substituting this into \eqref{eq_180321_03} after taking $x=t$, we get 
\begin{equation} \label{eq_180321_04}
\frac{d}{dt}(\Phi(t,t)) + \Phi(t,t)(\phi^+(t,t) + \phi^-(t,t))  =0.
\end{equation}
Therefore, 
$
\Phi(t,t) = C \exp\left( -\int_{0}^{t} (\phi^+(\tau,\tau) + \phi^-(\tau,\tau)) d\tau \right) = C m(t)^{-1} 
$
by \eqref{eq_180403_01}. 
To determine $C$, we take $x=t=0$ in equation \eqref{IE_01}. 
Then $\Phi(0,0)=1$, since the integral on the left-hand side 
is zero because $K(x)=0$ for $x<0$, 
and thus $C=1$ by $m(0)=1$ (since $\mathsf{K}[0]$ is the zero operator). 
Hence we obtain the first equality of \eqref{formula_01}. 
The second equality of \eqref{formula_01} is proved by the same way.   \hfill $\qed$

%
%

\section{Proof of Theorem \ref{thm_2}} 

%
%

%
Let $\psi(s)=\Gamma'(s)/\Gamma(s)$ be the digamma function. 
To study $g_\theta(x)$ from \eqref{eq_180323_03}, we need the following result. 

\begin{prop} \label{prop_1} For $\theta>0$, we define 
\begin{equation} \label{eq_180326_03}
\Psi_\theta^1(x)
:=
\frac{1}{2\pi}\lim_{U \to \infty} \int_{-U+ic}^{U+ic} \left[ s^\theta \cdot \exp\left(-\theta \psi(s)-\frac{\theta}{s}\right) -1 \right] e^{-izx} dz, 
\end{equation}
where $s=1/2-iz$. 
Then the right-hand side converges for $c>-1/2$ 
and the integral is independent of such $c$. 
$\Psi_\theta^1(x)$ is a real-valued continuous function on $(0,\infty)$, $\Psi_\theta^1(x)=0$ on $(-\infty,0)$, 
$\Psi_\theta^1(x) \ll 1$ as $x \to 0^+$, 
$\Psi_\theta^1(x) \ll \exp(-x/2)$ as $x \to +\infty$, 
and  
\begin{equation} \label{eq_180330_01}
s^\theta \cdot \exp\left(-\theta \psi(s)-\frac{\theta}{s}\right) -1  
= \int_{0}^{\infty} \Psi_\theta^1(x) e^{izx} dx, \quad s=\frac{1}{2}-iz
\end{equation}
holds for $\Im (z)>-1/2$.   
\end{prop}
\begin{proof} The integrand of \eqref{eq_180326_03} is holomorphic in $\Im z \geq -1/2$ except for $z=-i/2$, 
since $\psi(s)+s^{-1}=\psi(s+1)$ is holomorphic in $\Re(s)>-1$. 
The Stirling formula of $\log \Gamma(s)$ 
for $|\arg s| \leq \pi-\delta$ and $|s| \geq \delta$ (\cite[\S12.33, \S13.6]{WW}) 
derives the asymptotic expansion 
\begin{equation} \label{eq_180402_01}
\psi(s) = \log s - \frac{1}{2s} - \sum_{n=1}^{N-1}\frac{B_{2n}}{2ns^{2n}} +O(|s|^{-2N})
\end{equation}
in the same region. Using this with $\psi(s)+s^{-1}=\psi(s+1)$, we have 
\[
\scalebox{0.9}{$
\aligned 
s^\theta & \cdot \exp\left(-\theta \psi(s)-\frac{\theta}{s}\right) -1 \\
&= \exp\left( \theta \log \frac{s}{s+1} + \frac{\theta}{2(s+1)} +  \sum_{n=1}^{N-1}\frac{\theta B_{2n}}{2n(s+1)^{2n}} +O(|s|^{-2N}) \right)-1 \\
&= \exp\left( -\theta \sum_{n=1}^{\infty}\frac{1}{n(s+1)^n} + \frac{\theta}{2(s+1)}
 +  \sum_{n=1}^{N-1}\frac{\theta B_{2n}}{2n(s+1)^{2n}} +O(|s|^{-2N}) \right)-1 \\
&= -\frac{\theta}{2(s+1)} + \frac{\theta(3\theta-10)}{24(s+1)^2} + \sum_{n=3}^{2N-1} \frac{C_{n-1}(\theta)}{(s+1)^n} + O(|s|^{-2N})
\endaligned 
$}
\]
for some polynomials $C_n(\theta) \in \Q[\theta]$. Hence 
\begin{equation} \label{eq_180329_01}
\aligned 
\Psi_\theta^1(x)
& =  \left( -\frac{\theta}{2}u(x) 
+ \sum_{n=2}^{2N-1}\frac{C_{n-1}(\theta)}{(n-1)!}x^{n-1} \right)e^{-3x/2}\mathbf{1}_{[0,\infty)}(x)  +  \\
& \quad + \frac{1}{2\pi}\int_{-\infty+ic}^{\infty+ic} R_N(s) e^{-izx} dz \\ 
\endaligned 
\end{equation}
by \cite[p. 167]{Ob}, where $u(x)=1$ for $x>0$, $u(0)=1/2$, 
and 
\[
R_N(s) = \left( s^\theta \exp\left(-\theta \psi(s)-\frac{\theta}{s}\right) -1 \right)
- \left( \frac{-\theta}{2(s+1)} + \sum_{n=2}^{2N-1} \frac{C_{n-1}(\theta)}{(s+1)^n} \right)
\] 
which is holomorphic in $\Re(s)\geq 0$ except for $s=0$ and satisfies $R_N(s)=O(|s|^{-2N})$. 
This estimate 
enables us to move the path of integration as $\Im (z)=c \to +\infty$ if $N \geq 1$. 
Therefore, $\int_{-\infty+ic}^{\infty+ic} R_N(s) e^{-izx} dz=0$ for $x<0$ 
and the integral is bounded as $x \to 0^+$. 
On the other hand, by moving the path of integration as $\Im (z)=c \to -1/2$, 
$\int_{-\infty+ic}^{\infty+ic} R_N(s) e^{-izx} dz \ll \exp(-x/2)$ as $x \to +\infty$. 
Hence, \eqref{eq_180330_01} holds for $\Im (z)>-1/2$ by the Fourier inversion formula.  
Moreover, we find that $\Psi_\theta^1(x)$ is real-valued by considering \eqref{eq_180330_01} 
for pure-imaginary $z$. 
\end{proof}

\begin{prop} \label{prop_2} 
Let $\theta>0$. There exists a real-valued continuous function $\Psi_\theta$,  
continuously differentiable on $(0,\infty)$, such that
\begin{equation} \label{eq_180323_04}
\exp(-\theta \psi(s)) = \int_{0}^{\infty} \Psi_{\theta}(x) \,  e^{izx} \, dx, \quad 
s=\frac{1}{2}-iz
\end{equation}
holds for $\Im (z)>-1/2$,  $\Psi_\theta(x)=0$ for $x<0$, 
$\Psi_\theta(x) = \Gamma(\theta)^{-1}x^{\theta-1}+O(x^\theta)$ as $x \to 0^+$, and  
$\Psi_\theta(x) \ll \exp(-\kappa x)$ as $x \to \infty$ for any $\kappa<1/2$. 
\end{prop}
\begin{proof} Define $\Psi_{\theta,\alpha}^0(x)=0$ for $x<0$ and 
\begin{equation} \label{eq_180329_06}
\Psi_{\theta,\alpha}^0(x) = e^{-x/2}(x/\alpha)^{(\theta-1)/2} I_{\theta-1}(2\sqrt{\alpha x}\,)
\end{equation}
for $x \geq 0$, 
where $I_\nu(z)$ is the modified Bessel function of the first kind. 
Then $\Psi_{\theta,\alpha}^0(x)$ is continuously differentiable on $(0,\infty)$ and   
\begin{equation} \label{eq_180329_03}
\frac{1}{s^{\theta}}\exp\left(\frac{\alpha}{s}\right) 
= \int_{0}^{\infty} \Psi_{\theta,\alpha}^0(x)\, e^{izx} dx, \quad s=\frac{1}{2}-iz,  
\end{equation}
holds if $\Im (z)>-1/2$ and $\theta>0$ by \cite[p. 173]{Ob} 
(and the changing of variable $\log(1/x) \mapsto x$). 
The trivial equality 
\[
\exp(-\theta\psi(s)) 
= 
\frac{1}{s^\theta}\exp\left(\frac{\theta}{s}\right) 
+
\frac{1}{s^\theta}\exp\left(\frac{\theta}{s}\right) \cdot \left[ s^\theta \exp\left(-\theta\psi(s)-\frac{\theta}{s}\right)-1 \right] 
\]
implies that \eqref{eq_180323_04} holds for 
\begin{equation} \label{eq_180329_02}
\Psi_\theta(x) 
= \Psi_{\theta,\theta}^0(x) + \int_{0}^{x} \Psi_{\theta,\theta}^0(y)\Psi_\theta^1(x-y) \, dy.
\end{equation}
In fact, the integral on the right-hand side exists 
by Proposition \ref{prop_1} and 
\begin{equation} \label{eq_180329_04}
\Psi_{\theta,\theta}^0(x) = x^{\theta-1}(\Gamma(\theta)^{-1}+O(x))
\end{equation}
which is derived from the series expansion
\begin{equation} \label{eq_180326_01}
I_\nu(z) = \sum_{m=0}^{\infty}\frac{(z/2)^{2m+\nu}}{m!\Gamma(\nu+m+1)}. 
\end{equation}
Clearly, $\Psi_\theta(x)$ is continuous on $(0,\infty)$, vanishes on $(-\infty,0)$ 
and is continuously differentiable on $(0,\infty)$. 
Therefore it remains to show the upper bound for large $x>0$. 
By \eqref{eq_180329_04}, 
\[
\Psi_\theta(x) 
= \Psi_{\theta,\theta}^0(x) +O(x) = x^{\theta-1}(\Gamma(\theta)^{-1}+O(x))+O(x)
\]
as $x \to 0^+$. The asymptotic formula 
\[
I_\nu(z) = \frac{e^z}{\sqrt{2\pi z}}(1+O(|z|^{-1})), \quad |\arg z|<\pi/2, \quad |z| \to \infty
\]
derives 
\begin{equation} \label{eq_180326_02}
\Psi_{\theta,\alpha}^0(x) \ll x^{(2\theta-3)/4} 
\exp\left( -\frac{x}{2} + 2\sqrt{\theta x} \right) \ll \exp(-\kappa x)
\end{equation}
as $x \to +\infty$ for any $\kappa <1/2$. 
By Proposition \ref{prop_1}, $e^{x/2}\Psi_\theta^1(x)$ is uniformly bounded on $[0,\infty)$. 
Therefore, 
\[
\aligned 
\int_{0}^{x} & \Psi_{\theta,\theta}^0(y)\Psi_\theta^1(x-y) \, dy 
 = \int_{0}^{x} \Psi_{\theta,\theta}^0(y) e^{-(x-y)/2} \cdot e^{(x-y)/2} \Psi_\theta^1(x-y) \, dy \\
& \ll e^{-x/2} \int_{0}^{x}e^{y/2} |\Psi_{\theta,\theta}^0(y)| \, dy 
= e^{-x/2} \int_{0}^{x} e^{y/2}\Psi_{\theta,\theta}^0(y) \, dy , 
\endaligned 
\]
since $|I_\nu(z)|=I_\nu(z)$ for real $z$, $\theta$ by \eqref{eq_180326_01}. 
Using \eqref{eq_180326_01} again, 
\[
\aligned 
\int_{0}^{x} & e^{y/2} \Psi_{\theta,\theta}^0(y) \, dy
 = \int_{0}^{x} 
y^{\theta-1} \sum_{m=0}^{\infty}\frac{1}{m!\Gamma(\theta+m)}(\theta y)^{m} \, dy \\
& = \sum_{m=0}^{\infty}\frac{\theta^m }{m!\Gamma(\theta+m+1)} y^{m+\theta} 
 = (x/\theta)^{\theta/2}I_\theta(2\sqrt{\theta x}) = e^{x/2}\Psi_{\theta+1,\theta}^0(x). 
\endaligned 
\]
Hence $\int_{0}^{x} \Psi_{\theta,\theta}^0(y)\Psi_\theta^1(x-y) \, dy \ll \Psi_{\theta+1,\theta}^0(x)$ as $x \to \infty$ 
which implies the estimate $\Psi_\theta(x) \ll \exp(-\kappa x)$ by \eqref{eq_180326_02}. 
\end{proof}

{\bf Proof of Theorem \ref{thm_2} (K-ii), (K-iii)} \quad 
We have 
\begin{equation}
\exp\left(-\frac{\alpha}{s}\right) 
= \int_{0}^{\infty} J_0(2\sqrt{\alpha x})\, s \,e^{-sx} dx
\end{equation}
for $\Re(s)>0$ by \cite[p. 173]{Ob} (and the changing of variable $\log(1/x) \mapsto x$). 
Therefore, by integration by parts, 
\[
\aligned 
\exp\left(-\frac{\alpha}{s}\right) -1 
& = \int_{0}^{\infty} \frac{d}{dx} J_0(2\sqrt{\alpha x})\, e^{-sx} dx \\
& = -  \alpha \int_{0}^{\infty} \frac{J_1(2\sqrt{\alpha x})}{\sqrt{\alpha x}} e^{-sx} dx.  
\endaligned
\]
Combining this with \eqref{eq_180323_04} and the equality 
\[
\aligned 
\exp\left(-2\theta \frac{\gamma'(s)}{\gamma(s)}\right)
& =  \pi^\theta 
\exp\left( - \theta \psi(\tfrac{s}{2}+1) \right) \\
& \quad + \pi^\theta 
\exp\left( - \theta \psi(\tfrac{s}{2}+1) \right)
\left[ \exp\left( - \frac{2\theta}{s-1} \right) - 1 \right],  
\endaligned 
\]
we obtain 
\begin{equation} \label{eq_180403_02}
\aligned 
g_\theta(x) & = 2\pi^\theta e^{-\frac{3}{2}x} \Psi_{\theta}(2x) \\ 
& \quad - 4\theta \pi^\theta e^{-\frac{3}{2}x} \int_{0}^{x} \Psi_{\theta}(2(x-y))\frac{ e^{2y} J_{1}(2\sqrt{2\theta y})}{\sqrt{2\theta y}} \, dy. 
\endaligned 
\end{equation}
In particular, $g_\theta(x)$ is a real-valued continuous function on $\R$ vanishing on $(-\infty,0)$ 
by Proposition \ref{prop_2} and the assumption $\theta>1$. 
Thus formula \eqref{eq_180329_05} implies (K-iii). 
The Dirichlet series \eqref{eq_180323_02} converges absolutely for $\Re(s)>1$ (\cite[Prop. 3.9.5]{Ih} or \cite[Thm. 2]{IM01}). 
Therefore, \eqref{eq_180323_01} holds for $\Im(z)>1/2$ 
by Proposition \ref{prop_2} and the Fubini theorem. 
Moreover, we have $K_\theta(x) \ll \exp(x/2)$ 
by moving the path of integration in the inversion formula of \eqref{eq_180323_01} 
noting the growth of $\psi(s)$ and the non-vanishing of $\zeta(s)$ for $\Re(s) \geq 1$. \hfill $\qed$
\medskip

{\bf Proof of Theorem \ref{thm_2} (K-iv)} \quad 
By \eqref{eq_180329_05} and Proposition \ref{prop_2}, it remains to show that 
$|(d/dx)g_\theta(x)|$ is integrable around $x=0$.  
We have 
\[
\scalebox{0.95}{$
\aligned 
\frac{d}{dx}g_\theta(x) 
& = -3 e^{-\frac{3}{2}x} \left( 
\Psi_{\theta}(2x) - 2\theta \int_{0}^{x} \Psi_{\theta}(2(x-y))\frac{ e^{2y} J_{1}(2\sqrt{2\theta y})}{\sqrt{2\theta y}} \, dy \right) \\ 
& + 2 e^{-\frac{3}{2}x} \left( 
\frac{d}{dx} \Psi_{\theta}(2x) - 2\theta \int_{0}^{x} \frac{d}{dx} \Psi_{\theta}(2(x-y))\frac{ e^{2y} J_{1}(2\sqrt{2\theta y})}{\sqrt{2\theta y}} \, dy
\right)
\endaligned 
$}
\]
and 
\[
\aligned 
\frac{d}{dx}\Psi_\theta(x) 
&= \frac{d}{dx}\Psi_{\theta,\theta}^0(x) 
+ \int_{0}^{x} \frac{d}{dy} \Psi_{\theta,\theta}^0(y)\Psi_\theta^1(x-y) \, dy.
\endaligned 
\]
Applying the series expansion \eqref{eq_180326_01} to definition \eqref{eq_180329_06} of $\Psi_{\theta,\alpha}^0$, 
we easily find that $|(d/dx)\Psi_{\theta,\theta}^0(x)|$ is integrable around $x=0$ by the assumption $\theta>1$. 
Therefore, the above two equalities implies that $|(d/dx)g_{\theta}(x)|$ is integrable around $x=0$. \hfill $\qed$
\medskip

{\bf Proof of Theorem \ref{thm_2} (K-v)} \quad Put $s=1/2-i(u+iv)$. For $\delta>0$, 
the estimate 
\begin{equation} \label{eq_180330_02}
\exp\left(-2\theta \frac{\xi'}{\xi}(s)\right)
\ll \exp\left(-\theta \Re \psi(s/2)\right) 
\ll (1+v)^{-\theta}
\end{equation}
holds uniformly for $u \in \R$ and $v \geq 1/2+\delta$ 
with the implied constant depending only on $\delta>0$. 
On the other hand, it will be shown that 
\begin{equation} \label{eq_180330_03}
(\mathsf{F}\mathsf{K}_\theta f)(z) 
= \exp\left(-2\theta \frac{\xi'}{\xi}(s)\right) 
(\mathsf{F}f)(-z), \quad s=\frac{1}{2}-iz 
\end{equation} 
holds for $f \in L^2(-\infty,t)$ and $\Im (z)>1/2$, 
where $(\mathsf{F}f)(z)=\int_{-\infty}^{\infty}f(x)e^{izx}dx$.  
Then, (K-v) is proved by a way similar to the proof of \cite[Prop. 4.2]{Su} 
if we use \eqref{eq_180330_02} (resp. \eqref{eq_180330_03}) 
instead of (4.3) (resp. (3.3)) of \cite{Su}. 
Hence it remains to show that \eqref{eq_180330_03} holds 
for $f \in L^2(-\infty,t)$ and $\Im (z)>1/2$. 
\medskip

Let $f \in L^2(-\infty,t)$. 
Then, 
$(\mathsf{F}f)(-(u+iv)) = \int_{-\infty}^{t} f(x) e^{vx -iux} \, dx$
is defined if $v \geq 0$. On the other hand, 
$\mathsf{K}_{\theta}f$ is defined and has a support in $[-t,\infty)$ (but not necessarily $L^2(-t,\infty)$), 
since we see $(\mathsf{K}_\theta f)(x) = \int_{-x}^{t}K_\theta (x+y)f(y) \,dy$.  
Moreover, $K_\theta(x) \ll \exp(x/2)$ implies $(\mathsf{K}_{\theta}f)(x) \ll \exp(x/2)$, 
and hence  
$
(\mathsf{F}\mathsf{K}f)(u+iv)
= \int_{-t}^{\infty} (\mathsf{K}_{\theta}f)(x) e^{-vx +iux} \, dx
$
is defined if $v > 1/2$. As a consequence, the calculation 
\begin{equation*}
\aligned
({\mathsf F}\mathsf{K}_{\theta}f)(z) 
& = \int_{-\infty}^{\infty}  \int_{-\infty}^{\infty} K_{\theta}(x+y) \, e^{izx} \, dx \, f(y) \, dy \\
& =  \int_{-\infty}^{\infty} K_{\theta}(x) \, e^{izx} \, dx  \int_{-\infty}^{\infty} f(y) \, e^{-izy} \, dy \\
& = \exp\left(-2\theta \frac{\xi'}{\xi}(s)\right)  (\mathsf{F}f)(-z) 
\endaligned
\end{equation*}
is justified if $\Im (z)>1/2$. Hence we complete the proof. \hfill $\qed$

\section{Approximate formulas} 

In this section, we study approximate formulas of $\Psi_\theta(x)$ and $g_\theta(x)$ in a bounded range $0 \leq x \leq x_0$.  
This is because formulas \eqref{eq_180329_02} and \eqref{eq_180403_02} are not so explicit 
from a computational point of view, 
since $\Psi_\theta^1(x)$ is given by the inversion formula \eqref{eq_180326_03}. 

If $R(s) \ll |s|^{-N}$ as $|s| \to \infty$ in $\Re(s)>1$, 
\[
\frac{1}{2\pi}\int_{-\infty+ic}^{\infty+ic} R\left(\frac{1}{2}-iz\right) e^{-izx} dz
\ll  e^{cx_0} (\tfrac{1}{2}+c)^{1-N} 
\]
for $c>1/2$ and $0\leq x \leq x_0$, where the implied constant depends only on $R(s)$.  
Therefore, the asymptotic expansion 
\begin{equation} \label{eq_180402_02}
s^\theta \exp\left(-\theta \psi(s)-\frac{\theta}{s}\right) -1  
= \sum_{n=1}^{N-1}\frac{\tilde{C}_{n}(\theta)}{s^n} +O(|s|^{-N})
\end{equation}
for $|\arg s| \leq \pi-\delta$, $|s| \geq \delta$ 
derived from \eqref{eq_180402_01} 
implies that $\Psi_\theta^1(x)$ is well approximated by 
\begin{equation} 
\Psi_{\theta}^{1,N}(x) 
:= \mathbf{1}_{[0,\infty)}(x)e^{-x/2}\sum_{n=1}^{N-1}\frac{\tilde{C}_{n}(\theta)}{(n-1)!}x^{n-1} 
\end{equation}
in a bounded range $0\leq x \leq x_0$ if $N$ is sufficiently large, 
where $\tilde{C}_{1}(\theta)=-\theta/2$, $\tilde{C}_{2}(\theta)=\theta(3\theta+2)/24$, etc. 
More precisely, $\Psi_\theta^1(x) - \Psi_{\theta}^{1,N}(x) \ll e^{(1/2+\epsilon)x_0} (1+\epsilon)^{1-N}$ 
for $0 \leq x \leq x_0$, where $\epsilon>0$ is a given constant.  
Thus $\Psi_\theta(x)$ is well approximated by 
\begin{equation} \label{180822_01}
\Psi_{\theta}^N(x) 
:= \Psi_{\theta,\theta}^0(x) + \int_{0}^{x} \Psi_{\theta,\theta}^0(y)\Psi_{\theta}^{1,N}(x-y) \, dy
\end{equation}
in a bounded range $0\leq x \leq x_0$ if $N$ is sufficiently large, that is,  
\[
\aligned 
\Psi_{\theta}(x) - \Psi_{\theta}^N(x) 
& = \int_{0}^{x} \Psi_{\theta,\theta}^0(y)(\Psi_{\theta}^{1}(x-y)-\Psi_{\theta}^{1,N}(x-y)) \, dy \\
& \ll e^{(1/2+\epsilon)x_0} \int_{0}^{x_0} |\Psi_{\theta,\theta}^0(y)| \, dy \cdot (1+\epsilon)^{1-N}. 
\endaligned 
\]
For the integral on the right-hand side of \eqref{180822_01}, the trivial equality
\[
\frac{1}{s^\theta}\exp\left(\frac{\theta}{s}\right) \cdot 
\sum_{n=1}^{N-1}\frac{\tilde{C}_{n-1}(\theta)}{s^n}
 = \sum_{n=1}^{N-1}\tilde{C}_{n}(\theta)\frac{1}{s^{\theta+n}}\exp\left(\frac{\theta}{s}\right) \\
\]
combined with \eqref{eq_180329_03} and \eqref{eq_180402_02} gives 
\[
\int_{0}^{x} \Psi_\theta^0(y)\Psi_{\theta}^{1,N}(x-y) \, dy
 = \sum_{n=1}^{N-1} \tilde{C}_{n}(\theta) \Psi_{\theta+n,\theta}^0(x).
\]
Hence, by taking $\tilde{C}_{0}(\theta)=1$, we obtain
\[
\Psi_{\theta}^N(x) 
= 
\sum_{n=0}^{N-1}\tilde{C}_n(\theta) \Psi_{\theta+n,\theta}^0(x). 
\]
These sums are useful for calculating $\Psi_\theta(x)$ in computational ways.  

On the other hand, we note the decomposition 
\[
\aligned 
\exp &\left(-\theta\psi\left(\frac{s}{2}+1\right) - \frac{2\theta}{s-1}\right) \\
& = 
\frac{1}{w^\theta}\exp\left(\frac{\theta}{w}\right) 
\cdot 
w^\theta \exp\left(-\theta\psi(w)-\frac{\theta}{w}\right) 
\cdot 
\exp\left( - \frac{2\theta}{2w-3}\right),
\endaligned 
\]
where $w=(s+2)/2$, and the asymptotic expansion 
\[
w^\theta \exp\left(-\theta\psi(w)-\frac{\theta}{w}\right) 
\exp\left( - \frac{2\theta}{2w-3}\right) 
= 1 + \sum_{n=1}^{N-1}\frac{A_n(\theta)}{w^n} + O\left( |w|^{-N} \right)
\]
for large $w$ derived from \eqref{eq_180402_01}, 
where $A_1(\theta)=-3\theta/2$, $A_2(\theta)=\theta(27\theta-34)/24$, etc.  
Then, we find that $g_\theta(x)$ is well approximated by 
\[
g_\theta^N(x) = \Psi_{\theta,\theta}^2(x) + \sum_{n=1}^{N-1} A_n(\theta) \Psi_{\theta+n,\theta}^2(x), 
\]
with $\Psi_{\theta,\alpha}^2(x) := 2e^{-3x/2}\Psi_{\theta,\alpha}^0(2x)$  
in a bounded range $0 \leq x \leq x_0$, 
since $\int_{0}^{\infty} \Psi_{\theta,\alpha}^2(x) e^{izx} \, dx
=w^{-\theta}\exp(\alpha/w)$ with $w=(s+2)/2$ and $s=1/2-iz$  
holds by \eqref{eq_180329_03}. 
As before, ``well approximated'' means that 
$g_\theta(x)-g_\theta^N(x) \ll_{x_0, \epsilon} (1+\epsilon)^{-N}$ holds for $0 \leq x \leq x_0$. 


%

\bigskip \noindent
\\
Department of Mathematics, 
School of Science, \\
Tokyo Institute of Technology \\
2-12-1 Ookayama, Meguro-ku, 
Tokyo 152-8551, JAPAN  \\
Email: {\tt msuzuki@math.titech.ac.jp}

\end{document}